\documentclass[11pt,reqno]{amsart}

\usepackage{amssymb}
\usepackage[colorlinks=true,citecolor=black,linkcolor=black,urlcolor=blue]{hyperref}

\theoremstyle{plain}
\newtheorem{thm}{Theorem}
\newtheorem{corr}{Corollary}[thm]
\newtheorem{lem}{Lemma}[section]

\numberwithin{equation}{section}

\newcommand{\arxiv}[1]{\href{http://arxiv.org/abs/#1}{\texttt{arXiv:#1}}}

\newcommand{\eps}{\epsilon}
\newcommand{\tu}{\tau}
\newcommand{\al}{\alpha}
\newcommand{\M}{\mathbf{M}}
\newcommand{\frlog}{\frac{\log d}{d}}
\newcommand{\G}{\mathcal{G}_{n,d}}

\newcommand{\ER}{\mathrm{ER}(n,d/n)}

\newcommand{\E}[1]{\mathbb{E}\left[#1\right]}
\newcommand{\pr}[1]{\mathbb{P}\left[#1\right]}

\title[Percolation on random graphs]{Percolation with small clusters on random graphs}

\author{Mustazee Rahman}

\address[Mustazee Rahman]{Department of Mathematics\\
University of Toronto\\
40 St. George Street\\
Toronto\\
ON M5S 2E4\\
Canada}

\thanks{The author's research was supported by an NSERC CGS grant.}

\email[Mustazee Rahman]{mustazee@math.toronto.edu}

\date{}

\subjclass[2010]{05C30, 05C69, 05C80}
\keywords{random graphs, regular graph, independent set, induced forest, percolation}

\begin{document}

\begin{abstract}
Consider the problem of determining the maximal induced subgraph in a random $d$-regular graph
such that its components remain bounded as the size of the graph becomes arbitrarily large.
We show, for asymptotically large $d$, that any such induced subgraph has size density at most $2(\log d)/d$
with high probability. A matching lower bound is known for independent sets.
We also prove the analogous result for sparse Erd\H{o}s-R\'{e}nyi graphs.
\end{abstract}

\maketitle

\section{Introduction} \label{sec:intro}

A subset $S$ of a graph $G$ is a \emph{percolation set with clusters of size at most $\tu$} if all the components of
the induced subgraph $G[S]$ have size at most $\tu$. For instance, independent sets have clusters of size one.
We consider the following problem on random $d$-regular graphs
and Erd\H{o}s-R\'{e}nyi graphs of average degree $d$. Given a threshold $\tu$ what is the density, $|S|/|G|$,
of the largest percolation sets $S$ with clusters of size at most $\tu$ on the aforementioned graph ensembles?
We say $S$ is a percolation set with small clusters when we do not want to mention the parameter $\tu$ explicitly.

Edwards and Farr \cite{EF} study this problem for some general classes of graphs under the notion of graph fragmentability.
They consider a natural $\tu \to \infty$ version of the problem and provide upper and lower bounds on densities of percolation
sets with small clusters for bounded degree graphs. Their bound is sharp for the family of graphs with maximum degree 3,
and optimal, in a sense, for several families of graphs such as trees, planar graphs or graphs with a fixed excluded minor.
However, their bounds are not of the correct order of magnitude for \emph{random} $d$-regular graphs.

For random graphs the correct order of the density of percolation sets with small clusters can be deduced
from just considering the largest independent sets (that is, the $\tau = 1$ case).
Bollob\'{a}s \cite{Bol} proved that with high probability the density of the largest independent sets in
a random $d$-regular graph is at most $2(\log d)/d$ for $d \geq 3$. The same bound was proved
for Erd\H{o}s-R\'{e}nyi graphs of average degree $d$ by several authors (see \cite{Bolbook} Theorem 11.25).
Frieze and {\L}uczak \cite{Fri, FL} provided lower bounds of order $2(\log d - \log \log d)/d$ for large $d$.

Our main result is that relaxing the problem from independent sets to percolation sets with small clusters
provides no improvement to the maximum density for large $d$. Roughly
speaking, for both the aforementioned graph ensembles we prove that for any $\tu$ and large $d$,
the density of the largest percolation sets with clusters of size at most $\tu$ is bounded above by $2(\log d)/d$
with high probability. In fact, $\tau$ may be taken to be of linear order in the size of the graph.
Precise statements are in Section \ref{sec:results}.

\subsection{Preliminaries and terminology} \label{sec:notation}

Let $V(G)$ and $E(G)$ denote the set of vertices and edges of a graph $G$, respectively.
For an integer $\tu \geq 1$ define
$$
\al^{\tu}(G) = \max \left \{\frac{|S|}{|V(G)|} : S \subset V(G) \;\text{is a percolation set with clusters of size at most}\; \tu \right \}.
$$
We say that a sequence of events $E_n$, generally associated to $\G$, occurs with high probability if
$\pr{E_n} \to 1$ as $n \to \infty$.

We use the configuration model (see \cite{Bolbook} chapter 2.4) to
sample a random $d$-regular graph $\G$ on $n$ labelled vertices.
Recall that $\G$ is sampled in the following manner. Each of the $n$ distinct vertices emit
$d$ distinct half-edges, and we pair up these $nd$ half-edges uniformly at random. (We tactically
assume that $nd$ is even.) These $nd/2$ pairs of half-edges can be glued into full edges
to yield a random $d$-regular graph. There are $(nd-1)!! = (nd-1)(nd-3)\cdots 3\cdot 1$ such graphs.

The resulting random graph $\G$ may have loops and multiple edges, that is, it is a multigraph.
However, the probability that $\G$ is a simple graph is uniformly bounded away from zero
at $n \to \infty$. In fact, Bender and Canfield \cite{BC} and Bollob\'{a}s \cite{Bol2} showed that
$$\pr{\G \;\text{is simple}} \underset{n \to \infty}{\longrightarrow}  e^{ \frac{1 - d^2}{4}}.$$

Also, conditioned on $\G$ being simple its distribution is a uniform $d$-regular simple graph on $n$ labelled vertices.
It follows from these observations that any sequence of events that occur with high probability for $\G$ (as $n \to \infty$)
also occurs with high probability for a uniformly chosen simple $d$-regular graph.

We denote by $\mathrm{ER}(n,p)$ an Erd\H{o}s-R\'{e}nyi graph on $n$ vertices
and edge inclusion probability $p$. In this model every pair of vertices $\{u,v\}$ is
independently included as an edge with probability $p$.
We are interested in the sparse case when $p = d/n$ for a fixed $d$.

We set the function $h(x) = -x\log(x)$ for $0 \leq x \leq 1$ with
the convention that $h(0) = 0$. We will use the following properties of $h(x)$ throughout.
\begin{align} \label{eqn:hproperty}
(1) & \quad h(xy) = xh(y) + yh(x). \\
(2) & \quad h(1-x) \geq x - x^2/2 - x^3/2 \;\;\text{for}\; 0 \leq x \leq 1. \nonumber \\
(3) & \quad h(1-x) \leq x - x^2/2 \;\;\text{for}\;\; 0 \leq x \leq 1. \nonumber
\end{align}

The inequalities in (\ref{eqn:hproperty}) follow from Taylor expansion. It is clearly valid for $x = 1$.
For $0 \leq x < 1$ note that $-\log(1-x) = \sum_k x^k/k$. Hence, $-\log(1-x) \geq x + x^2/2$,
which implies that $h(1-x) \geq x - (1/2)x^2 - (1/2)x^3$. Furthermore,
$-\log(1-x) \leq x + (1/2)x^2 + (1/3)x^3 (1 + x + x^2 \cdots)$, which shows that
$-\log(1-x) \leq x + (1/2)x^2 + x^3/(3(1-x))$ for $0 \leq x < 1$. Consequently,
$h(1-x) \leq x - (1/2)x^2 - (1/6)x^3 \leq  x - (1/2)x^2$.

\subsection{Statement of results} \label{sec:results}

\begin{thm} \label{thm:RRGperc}
Let $\tau = \eps_d \frac{\log d}{d} \, n$ where $0 < \eps_d \leq 1$ and $\eps_d \to 0$ as $d \to \infty$.
Given $\eps > 0$ there exists a $d_0 = d_0(\eps, \{\eps_d\})$ such that if $d \geq d_0$, then with high probability any
induced subgraph of $\G$ with components of size at most $\tau$ has size at most
$$(2 + \eps) \frac{\log d}{d}\,n\,.$$
\end{thm}

\begin{corr} \label{thm:dreg}
For $\eps > 0$ and every fixed $\tu$ with respect to $n$
there exits a $d_0 = d_0(\eps)$ such that for $d \geq d_0$,
$$ \pr{\al^{\tu}(\G) \leq (2+\eps)\frlog} \to 1 \quad \text{as}\; n \to \infty\,.$$
\end{corr}

It can be verified with careful bookkeeping in the proof of
Theorem \ref{thm:RRGperc} that for every such fixed $\tau$,
$\al^{\tu}(\G) \leq \frac{2(\log d +2 - \log 2)}{d}$ with high probability 
if $d \geq 12$. For Erd\H{o}s-R\'{e}nyi graphs we provide a weaker
but more explicit result.

\begin{thm} \label{thm:ERpercolation}
For $d \geq 5$, let $\tau = \log_d(n) - \log \log \log (n) - \log (\omega_n)$
where $\omega_n \to \infty$ with $n$. With high probability
any induced subgraph of $\ER$ with components of size at most $\tau$ has size at most
$$ \frac{2}{d} \left( \log d + 2 - \log 2 \right)n.$$
\end{thm}

\begin{corr} \label{thm:ER}
If $\al_{\rm{ER}(d)} = \frac{2}{d}(\log d + 2 - \log 2)$
then for every fixed $\tu$ with respect to $n$,
$$ \pr{\al^{\tu}(\ER) \leq \al_{\rm{ER}(d)}} \to 1 \quad \text{as}\; n \to \infty\,.$$
\end{corr}
 
We provide another interpretation of Corollaries \ref{thm:dreg} and \ref{thm:ER}.
Bayati, Gamarnik and Tetali \cite{BGT} proved that the quantities $\al^{1}(\G)$
and $\al^{1}(\ER)$ converge almost surely to non-random limits as $n \to \infty$.
Their argument can be used to show that $\al^{\tu}(\G)$ and $\al^{\tu}(\ER)$ also
converge almost surely, as $n \to \infty$, to non-random limits $\al^{\tu}(d)$ and $\al^{\tu}(\mathrm{ER}(d))$,
respectively.

It is thus natural to consider the limiting values of
$\al^{\tu}_d$ and $\al^{\tu}(\rm{ER}(d))$ as $\tu \to \infty$. Define
$$\al^{\infty}(d) = \sup_{\tu} \al^{\tu}(d) \;\;\text{and}\;\; \al^{\infty}(\rm{ER}(d)) = \sup_{\tu} \al^{\tu}(\rm{ER}(d)).$$
In a sense these parameters determine the largest size density of percolation sets in $\G$
and $\ER$ whose components remain bounded as $n \to \infty$.
Corollaries \ref{thm:dreg} and \ref{thm:ER} along with the matching lower bound
of Frieze and {\L}uczak \cite{Fri, FL} imply that
\[ \lim_{d \to \infty} \frac{\al^{\infty}(d)}{(\log d)/d} = 2 \quad
\text{and} \quad \lim_{d \to \infty} \frac{\al^{\infty}(\rm{ER}(d))}{(\log d)/d} = 2.\]

We briefly discuss what is known about $\al^{\tu}(d)$ and $\al^{\infty}(d)$
for small values of $d$. For independent sets,
McKay \cite{McKay} proved that $\al^1(3) \leq 0.4554$ and this bound
was recently improved by Barbier et al.~\cite{BKZZ} to $\al^1(3) \leq 0.4509$.
Cs\'{o}ka et al.~\cite{CGHV} showed by way of randomized algorithms that
$\al^1(3) \geq 0.4361$ and this was improved to $\al^1(3) \geq 0.4375$
by Hoppen and Wormald \cite{HW}.

Hoppen and Wormald \cite{HW08} also provide a lower bound to the largest size
density of an induced forest in $\G$, and their construction can be used to get the same lower bound for
$\al^{\infty}(d)$. An upper bound to the density of induced forests was given by Bau et al.~\cite{BWZ}
with numerical values for small $d$. These upper bounds hold true for $\al^{\infty}(d)$ as well.
On the other hand it is known that $\al^{\infty}(3) = 3/4$ through results on the
fragmentability of graphs by Edwards and Farr \cite{EF}, and it is conjectured in \cite{BWZ}
that $\al^{\infty}(4) = 2/3$.
 
The question of the size density of the largest induced forests in $\G$ can also be treated
with the techniques used to prove Theorem \ref{thm:RRGperc}. The proof of the
theorem can be used with little modification to show that for large $d$, the size
density of the largest induced forests in $\G$ is also at most $(2+o(1))\frlog$ with high probability.
The same conclusion holds for the size density of the largest $k$-independent
sets in $\G$ for every fixed $k$. (A $k$-independent set is a subset of vertices such
that the induced subgraph has maximum degree $k$.)
 
We prove Theorem \ref{thm:RRGperc} in Section \ref{sec:regular} and Theorem \ref{thm:ERpercolation} in Section \ref{sec:ER}.

\section{Percolation on random regular graphs} \label{sec:regular}

The proof of Theorem \ref{thm:RRGperc} is based on the following two lemmas.
In the following we prove Theorem \ref{thm:RRGperc} by using these lemmas.
The lemmas are then proved in Section \ref{sec:ksparse} and Section \ref{sec:regulardensitybound}, respectively.

A finite (multi)-graph $H$ is $k$-sparse if
$|E(H)|/|V(H)| \leq k$, that is, the average degree of $H$ is at most $2k$.
For example, finite trees are 1-sparse.
Any subgraph of a $d$-regular graph is $(d/2)$-sparse.
The first lemma shows that linear sized subgraphs of a random $d$-regular graph
are likely to be $k$-sparse so long as their size density is sufficiently small.

\begin{lem} \label{lem:ksparse}
Let $\G$ be a random $d$-regular graph on $n$ vertices. Suppose $d \geq 12$ and $3.5 < k \leq (1- \frac{1}{\sqrt{2}})d$.
Set $C_{k,d} = e^{-4}(2k/d)^{1 + \frac{1}{k-1}}$. With high probability, any subgraph in $\G$ of size
at most $C_{k,d}\cdot n$ is $k$-sparse. The probability that this property fails in $\G$ is $O_{k,d}(n^{3.5-k})$.
\end{lem}

The next  lemma shows that $k$-sparse subgraphs of $\G$ are actually not
very large if $k = o(\log d)$ and $d$ is sufficiently large.

\begin{lem} \label{lem:densitybound}
Let $\G$ be a random $d$-regular graph on $n$ vertices. Let $k = \eps_d \log d$ where $0 < \eps_d \leq 1$
and $\eps_d \to 0$ as $d \to \infty$. Given any $\eps > 0$ there is a $d_0 = d_0(\eps, \{\eps_d\})$ such
that if $d \geq d_0$, then with high probability any $k$-sparse induced subgraph of $\G$ has size at most
$$ (2 + \eps) \frac{\log d}{d}\,n\,.$$
\end{lem}

We do not attempt to provide explicit upper bounds on $d_0$.

\paragraph{\textbf{Proof of Theorem} \ref{thm:RRGperc}}

Let $\eps_d$ be as in the statement of the theorem.
First we show that it is possible to choose $k \geq 4$
satisfying both the constraints that $k = o(\log d)$ and
$e^{-4} (2k/d)^{1 + 1/(k-1)} \geq \eps_d \frac{\log d}{d}$ for all large $d$.
Let $\eps'_d = \max \{ \frac{4}{|\log(\eps'_d)|}, \frac{4}{\log d} \}$.
Note that $\eps'_d \to 0$ as $d \to \infty$. We assume that $d$ is large enough
that $\eps_d \leq e^{-6}$. Set $k = \eps'_d \log d$.

We begin by showing that $e^{-4}(2k/d)^{1+ 1/(k-1)} \geq \eps_d \frac{\log d}{d}$
for all large $d$. As $(2k/d) \leq 1$ we have
$$(\frac{2k}{d})^{1+ 1/(k-1)} \geq (\frac{2k}{d})^{1 + 2/k} \geq (\frac{k}{d})^{1 + 2/k} \geq
\big (\frac{\eps'_d \log d}{d} \big)^{1 + \frac{2}{\eps'_d \log d}}.$$

We now show that $(\frac{\eps'_d \log d}{d})^{\frac{2}{\eps'_d \log d}} \geq \eps_d^{1/2}$,
which would imply that the very last term above is greater than $(\eps'_d \eps_d^{1/2}) \frac{\log d}{d}$.
Observe that

\begin{align*}
\log \Big ( (\frac{\eps'_d \log d}{d})^{\frac{2}{\eps'_d \log d}}\Big) & = \frac{2}{\eps'_d} \Big( \frac{\log(\eps'_d) + \log\log d}{\log d} - 1\Big)\\
& \geq \frac{2}{\eps'_d} \Big ( \frac{\log 4}{\log d} - 1 \Big) \quad (\text{as}\; \eps'_d \geq \frac{4}{\log d}) \\
& \geq \frac{-2}{\eps'_d}.
\end{align*}

We conclude that $(\frac{\eps'_d \log d}{d})^{\frac{2}{\eps'_d \log d}} \geq e^{-2/\eps'_d}$,
and as $\eps'_d \geq \frac{4}{|\log (\eps_d)|}$, we deduce that
$e^{-2/\eps'_d} \geq e^{-\frac{1}{2}|\log(\eps_d)|} = \eps_d^{1/2}$.

So far we have seen that $(2k/d)^{1 + 1/(k-1)} \geq (\eps'_d \eps_d^{1/2}) \frac{\log d}{d}$
for all large $d$ (large $d$ is required to ensure that $\eps'_d \leq 1$). Now we show that
$e^{-4} \eps'_d \, \eps_d^{1/2} \geq \eps_d$ for perhaps larger $d$.
As $\eps'_d \geq \frac{4}{|\log(\eps_d)|}$, this holds if $|\log(\eps_d)| \leq 4e^{-4} \eps_d^{-1/2}$.
Since $\eps_d \leq 1$, this inequality is the same as $\log(\eps_d^{-1}) \leq 4e^{-4}\eps_d^{-1/2}$.
A simple calculation shows that $\log(x) \leq 4e^{-4} x^{1/2}$ if $x \geq e^8/4$.
Therefore, $e^{-4} \eps'_d \, \eps_d^{1/2} \geq \eps_d$ whenever $\eps_d \leq 4e^{-8}$.
The latter certainly holds for large $d$.

We have thus concluded that it is possible to choose $k \geq 4$ satisfying both the constraints
that $k = o(\log d)$ and $e^{-4} (2k/d)^{1 + 1/(k-1)} \geq \eps_d \frac{\log d}{d}$ for all large $d$.
We are now able to finish the proof. Set $k = \eps'_d \log d$ in the following.

Let $A = A(\G)$ be the event that all subgraphs of $\G$ containing at most $\eps_d \frac{\log d}{d} n$
vertices are $k$-sparse. From the conclusion derived above we see that there exists $d_1$ such that
if $d \geq d_1$ then $k/d < 1 - 1/\sqrt{2}$ and $e^{-4}(2k/d)^{1 + 1/(k-1)} \geq \eps_d \frac{\log d}{d}$. 
Lemma \ref{lem:ksparse} implies that $\pr{A} \to 1$ as $n \to \infty$.

Let $B = B(\G)$ be the event that any induced subgraph of $\G$ that is $k$-sparse
contains at most $(2+ \eps) \frac{\log d}{d}\,n$ vertices. From Lemma \ref{lem:densitybound}
we conclude that there exists a $d_2$ such that if $d \geq d_2$ then $\pr{B} \to 1$ as $n \to \infty$.

If $d \geq \max \{d_1,d_2\}$ then $\pr{A\cap B} \geq \pr{A} + \pr{B} - 1 \to 1$ as $n \to \infty$.
Let $D = D(\G)$ be the event that all induced subgraphs of $\G$ with components of size
at most $\tau = \eps_d \frac{\log d}{d}\,n$ have size at most $(2 + \eps) \frac{\log d}{d} \, n$.
We show that $A \cap B \subset D$ for all $d \geq \max \{ d_1, d_2\}$.

Suppose a $d$-regular graph $G$ on $n$ vertices satisfies properties $A$ and $B$.
If $S \subset V(G)$ induces a subgraph with components of size at most $\tau = \eps_d \frac{\log d}{d} n$
then all components of $S$ are $k$-sparse because $G$ satisfies property $A$.
Hence, $S$ itself induces a $k$-sparse subgraph. As $G$ also satisfies property $B$
we deduce that $S$ contains at most $(2 + \eps) \frac{\log d}{d}\, n$ vertices.
This means that $G$ satisfies property $D$, as required.

The proof of Theorem \ref{thm:RRGperc} is now complete because if
$d \geq \{d_1,d_2\}$ then $\pr{D} \geq \pr{A \cap B} \to 1$ as $n \to \infty$.

\subsection{Proof of Lemma \ref{lem:ksparse}} \label{sec:ksparse}

We prove Lemma \ref{lem:ksparse} by showing that the expected number subgraphs of $\G$
that are of size at most $C_{k,d}\cdot n$ and that are not $k$-sparse is vanishingly small
as $n \to \infty$. The first moment bound implies that the probability is vanishingly small as well.

Let $Z_{i,j} = Z_{i,j}(\G)$ be the number of subsets $S \subset V(\G)$ such that $|S| = i$ and $e(S) = j$.
Notice that $Z_{i,j} = 0$ unless $j \leq (d/2)i$.

Let $N$ be the number of subgraphs of $\G$ that have size at most $C_{k,d}\cdot n$ and that
are not $k$-sparse. We have
\begin{equation} \label{eqn:Nformula}
N = \sum_{i=1}^{C_{k,d}n} \sum_{j = ki}^{(d/2)i} Z_{i,j}\,.
\end{equation}

In the following sequence of lemmas we compute $\E{Z_{i,j}}$ in order to bound to $\E{N}$.

\begin{lem} \label{lem:zijexpectation}
For $1 \leq i \leq n$ and $0 \leq j \leq (d/2)i$, the expectation of $Z_{i,j}$ is
\begin{equation}  \label{eqn:zijexpectation}
\E{Z_{i,j}} = \binom{n}{i} \times \dfrac{(id)! \; \big((n-i)d\big)! \; (nd/2)! \; 2^{id -2j}}{(id -2j)!\; j!\; \Big(\frac{nd}{2} -id + j\Big)! \; (nd)!}.
\end{equation}
\end{lem}

\begin{proof}
There are $\binom{n}{i}$ subsets $S$ of size $i$ and $\E{Z_{i,j}}$ is the sum over each such $S$
of the probability that $e(S) = j$. For a fixed subset $S$ of size $i$, the probability that
$e(S) = j$ is the number of pairings in the configuration model that satisfy $e(S)=j$ divided by $(nd-1)!!$.
The number of such pairings is
\[ \binom{id}{id - 2j} \binom{(n-i)d}{id -2j} \; (id -2j)! \; (2j-1)!! \; \Big((n-2i)d +2j -1\Big)!!\,.\]
Therefore, $\E{Z_{i,j}}$ equals
\begin{equation} \label{eqn:Zijformula}
\E{Z_{i,j}} = \binom{n}{i} \times \dfrac{\binom{id}{id - 2j} \binom{(n-i)d}{id -2j} \; (id -2j)! \; (2j-1)!! \; \Big((n-2i)d +2j -1\Big)!!}{(nd-1)!!}.
\end{equation}

We may simplify (\ref{eqn:Zijformula}) by using $(m-1)!! = \frac{m!}{2^{m/2}(m/2)!}$
for even integers $m \geq 2$ and $0!! = 1$. This simplification leads to (\ref{eqn:zijexpectation}).
\end{proof}

\begin{lem} \label{lem:Zijratio}
Suppose $1 \leq k \leq d/2$ and $1 \leq i \leq (2k/d)n$. For $ki \leq j \leq (d/2)i$, $\E{Z_{i,j}}$ is
maximized at $j = ki$.
\end{lem}

\begin{proof}
From the equation for $\E{Z_{i,j}}$ in (\ref{eqn:zijexpectation}) we deduce that the ratio
\[\frac{\E{Z_{i,j+1}}}{\E{Z_{i,j}}} = \dfrac{(id -2j-1)(id -2j)}{4 (j+1)\Big(\frac{(n-2i)d}{2} + j + 1\Big)}\,.\]

If $i \leq (2k/d)n$ then this ratio is at most 1 provided that $ki \leq j \leq (id)/2$.
Indeed, subtracting the denominator from the numerator gives
$id(id-1) -2(n-2i)d - 4 - 2j(nd + 3)$. This is non-positive for all $ki \leq j \leq (id)/2$ if and only if
\begin{equation} \label{eqn:ratiobound} ki \geq \dfrac{\frac{1}{2}(id)(id -1) - (n-2i)d - 2}{nd + 3}\,.\end{equation}

In order to show that (\ref{eqn:ratiobound}) holds for $1 \leq i \leq (2k/d)n$ it suffices to show that
$ki \geq \frac{(id)^2}{2nd}$ because the latter term is larger than the right hand side
of inequality (\ref{eqn:ratiobound}). Since $i \geq 1$, $ki \geq \frac{(id)^2}{2nd}$ if and
only if $k \geq \frac{id}{2n}$, which is indeed assumed.
\end{proof}

It follows from Lemma \ref{lem:Zijratio} and (\ref{eqn:Nformula}) that
\begin{align}
\E{N} &\leq \sum_{i=1}^{C_{k,d}n} (id/2) \E{Z_{i,ki}} \nonumber \\
& \leq dn^2 \max_{1 \leq i \leq C_{k,d}n} \E{Z_{i,ki}}. \label{eqn:Nbound}
\end{align}

To get a bound on $\E{Z_{i,ki}}$ that is suitable for asymptotic analysis we introduce some notation.
For a graph $G$ and subsets $S,T \subset V(G)$ let
$$m(S,T) = \dfrac{\left |(u,v): u \in S, v \in T, \{u,v\} \in E(G) \right|}{2|E(G)|} \,.$$
The \emph{edge profile} of $S$ associated to $G$ is the $2 \times 2$ matrix
$$
M(S) = \left [ \begin{array}{cc}
m(S,S) & m(S,S^{c}) \\
m(S^{c},S) & m(S^{c},S^{c})
\end{array} \right ]
$$
where $S^{c} = S \setminus V(G)$.
If $|S| = i$ and $e(S) = j$ then
\begin{equation} \label{eqn:edgeprofile}
M(S) = \left [ \begin{array}{cc}
\frac{2j}{nd} & \frac{i}{n} - \frac{2j}{nd} \\
 \frac{i}{n} - \frac{2j}{nd} & 1 - 2\frac{i}{n} + \frac{2j}{nd}
\end{array} \right ]
\end{equation}

We denote the matrix in the r.h.s.~of (\ref{eqn:edgeprofile}) by $M(i/n,j/(nd))$. Then
$Z_{i,j}$ is the number of $S \subset V(\G)$ such that $M(S) = M(i/n,j/(nd))$.
The \emph{entropy} of a finitely supported probability distribution $\pi$ is
\begin{equation} \label{eqn:entropydef}
H(\pi) = \sum_{x \in {\rm support}(\pi)} -\pi(x)\log \pi(x)\,.
\end{equation}

\begin{lem} \label{lem:regexpectation}
For $1 \leq i \leq n-1$ and $0 \leq j \leq id/2$, we have that
$$\E{Z_{i,j}} \leq O(d \sqrt{n}) \times \exp{\left \{ n \left [ \frac{d}{2}H\Big(M(i/n,j/(nd))\Big) - (d-1)H\big(i/n,1- (i/n)\big) \right ]\right \}}$$
where big O constant is universal.
\end{lem}

\begin{proof}
We use Stirling's approximation of $m!$ to simplify (\ref{eqn:zijexpectation}):
$$1 \leq \dfrac{m!}{\sqrt{2 \pi m} (m/e)^m} \leq e^{1/12m}.$$
First, consider $\binom{n}{\al n}$. For $1 \leq i \leq n-1$, Stirling's approximation
shows that $\binom{n}{i} \leq \sqrt{n/i(n-i)} \,e^{nH(i/n,1-i/n)}$. Since $n/i(n-i) \leq n/(n-1) \leq 2$
and $H(0,1) = H(1,0) = 0$, we conclude that $\binom{n}{\al n} \leq 2\, e^{nH(\al,1-\al)}$.

Now consider the fraction in (\ref{eqn:zijexpectation}), which is the probability
that $e(\{1, \ldots, i\}) = j$ in $\G$. Stirling's approximation implies
that the polynomial order term (in $n$) for this fraction is bounded from above,
up to an universal multiplicative constant, by
\begin{equation} \label{eqn:stirlingpolyterms}
\left [ \dfrac{d(nd/2)}{(id-2j)j((nd/2)-id+j)} \right ]^{1/2}.
\end{equation}

We may assume that each of the terms $id-2j$, $j$ and $(nd/2)-id+j$ are positive
integers. For if one of these were zero then the corresponding factorial in (\ref{eqn:zijexpectation})
would be 1 and we could ignore that term from the calculation. So  $(id-2j)j((nd/2)-id+j) \geq 1$,
which implies that (\ref{eqn:stirlingpolyterms}) is bounded above by $d \sqrt{n}$.

The term of exponential order (in $n$) for the fraction in (\ref{eqn:zijexpectation}) is
$$\dfrac{(id)^{id} \, \big((n-i)d \big )^{(n-i)d} (nd)^{nd/2}}
{(id -2j)^{id -2j}\; (2j)^{j}\; \Big((n-2i)d + 2j\Big)^{((nd/2)-id) + j} \; (nd)^{nd}}\,.$$
This may be written in exponential form as
\begin{align*}
 & \left [ \dfrac{(i/n)^{(i/n)} (1-(i/n))^{1-(i/n)}}{\big((i/n) - \frac{2j}{nd}\big)^{(i/n) - (2j/nd)} \,
 \big(\frac{2j}{nd}\big)^{j/nd} \, \big(1-2(i/n) + \frac{2j}{nd}\big)^{1/2 - (i/n) + j/nd}} \right ]^{nd}\\
 =& \, \exp{\left \{ n \left [\frac{d}{2}H\Big(M(i/n,j/(nd))\Big) - dH\big(i/n,1-(i/n) \big) \right] \right \}}\,.
\end{align*}

Therefore, (\ref{eqn:zijexpectation}) is bounded from above by
$$O(d\sqrt{n}) \exp{ \left \{ n \left [ \frac{d}{2}H\Big(M(i/n,j/(nd))\Big) - (d-1)H\big(i/n,1- (i/n)\big) \right ]\right \}}.$$
\end{proof}

As we want to bound $\E{Z_{i,ki}}$ we analyze of the maximum of
$(d/2)H(M(i/n,ki/nd)) - (d-1)H(i/n,1-(i/n))$ over the range $1 \leq i \leq C_{k,d} \cdot n$.
Lemma \ref{lem:regexpectation} implies that $\E{Z_{i,ki}}$ is bounded from above by
$$O(d\sqrt{n}) \times \exp{\left \{ n[(d/2)H(M(i/n,ki/nd)) - (d-1)H(i/n,1-(i/n))] \right \}}.$$

It is convenient to work with the analytic continuation of the terms involving the entropy.
Recall that $h(x) = -x\log x$. If we set $\al = i/n$ then $(d/2)H(M(i/n,ki/nd)) - (d-1)H(i/n,1-(i/n))$ equals
\begin{equation} \label{eqn:entterm}
(d/2)[h(\al(2k/d)) + 2h(\al - \al(2k/d)) + h(1-2\al + \al(2k/d))] - (d-1)H(\al,1-\al).
\end{equation}

Here $\al$ lies in the range $1/n \leq \al \leq C_{k,d}$.
We will show that (\ref{eqn:entterm}) is decreasing in $\al$ if $0 \leq \al \leq C_{k,d}$.
We will then evaluate its value at $\al = 1/n$ to show that the leading term (in $n$)
is $(1-k)(\log n)/n$. This will allow us to conclude Lemma \ref{lem:ksparse}.

\begin{lem} \label{lem:decreasingent}
Suppose that $2 \leq k \leq (1- 1/\sqrt{2})d$. Then the entropy term in (\ref{eqn:entterm}) is
decreasing as a function of $\al$ for $0 \leq \al \leq C_{k,d}$.
\end{lem}

\begin{proof}
We differentiate (\ref{eqn:entterm}) to show that it is negative for $0 < \al < C_{k,d}$.
Notice that the derivative $h'(\al) = - 1 - \log(\al)$. Differentiating (\ref{eqn:entterm}) in $\al$ and simplifying gives
$$\frac{d}{2} \Big( h(\frac{2k}{d}) + 2h(1-\frac{2k}{d})\Big) +
(k-1)\log(\al) + (d-1)(-\log(1-\al)) -(d-k)(-\log(1-2\al + \frac{2k}{d}\al))\,.$$

First, we deal with the term $(d-1)(-\log(1-\al)) -(d-k)(-\log(1-2\al + \frac{2k}{d}\al))$
and show that it is negative for $0 < \al < 1/2$. We will use the following inequalities for $-\log(1-x)$
which can be deduced from Taylor expansion. If $0 \leq x \leq 1/2$ then
$-\log(1-x) \leq x +(1/2)x^2 + (2/3)x^3$. If $0 \leq x \leq 1$ then $-\log(1-x) \geq x + (1/2)x^2 + (1/3)x^3$.
From these inequalities we conclude that $(d-1)(-\log(1-\al)) -(d-k)(-\log(1-2\al + \frac{2k}{d}\al))$
is bounded from above by
$$(d-1)(\al + \frac{\al^2}{2} + \frac{\al^3}{3}) -(d-k)[2(1-\frac{k}{d})\al +2 (1-\frac{k}{d})^2\al^2
+ \frac{8}{3}(1-\frac{k}{d})^3\al^3]\,.$$

The term $(1-\frac{k}{d})$ is positive and decreasing in $k$ if $2 \leq k \leq (1-1/\sqrt{2})d$.
Its minimum value is $1/\sqrt{2}$. Thus, $(1-\frac{k}{d})^2 \leq 1/2$ and $(1-\frac{k}{d})^3 \leq 1/\sqrt{8}$.
We deduce from this that
\begin{align*}
(d-1)(\al + \frac{\al^2}{2} + \frac{\al^3}{3}) -(d-k)[2(1-\frac{k}{d})\al +2 (1-\frac{k}{d})^2\al^2 + \frac{8}{3}(1-\frac{k}{d})^3\al^3] &\leq\\
(d-1)(\al + \frac{\al^2}{2} + \frac{\al^3}{3}) - \frac{d}{\sqrt{2}}[\sqrt{2}\al + \al^2 + \frac{\sqrt{8}}{3}\al^3] & =\\
- \al - \frac{(\sqrt{2}-1)d + 1}{2}\al^2 - \frac{d+3}{3}\al^3 &.
\end{align*}
The last term is clearly negative for positive $\al$. This shows what we had claimed.

Now we consider the term $\frac{d}{2} \big( h(\frac{2k}{d}) + 2h(1-\frac{2k}{d})\big) + (k-1)\log(\al)$
and show that it is negative for $0 < \al < C_{k,d}$. By property (2) of $h(x)$ from
(\ref{eqn:hproperty}) we have $h(1-x) \leq x$. Therefore, $h(1-\frac{2k}{d}) \leq 2k/d$ and
$(d/2)[h(\frac{2k}{d}) + 2h(1-\frac{2k}{d})] \leq k\log(d/2k) + 2k$. Thus,
$$\frac{d}{2} \big( h(\frac{2k}{d}) + 2h(1-\frac{2k}{d})\big) + (k-1)\log(\al) \leq k\log(d/2k) + 2k + (k-1)\log(\al).$$

The latter term in increasing in $\al$ because $k \geq 2$ and it tends to $-\infty$ as $\al \to 0$.
It is therefore negative until its first zero, which is the value $\al^{*}$ satisfying
$- \log(\al^{*}) = \frac{ k \log(d/2k) + 2k}{k-1}$. Observe that
$\frac{ k\log(d/2k) + 2k}{k-1} \leq (1 + \frac{1}{k-1})\log(d/2k) + 4$ since $k \geq 2$.
Consequently, $\al^{*} \geq e^{-4}(2k/d)^{1 + 1/(k-1)}$ and we conclude that
$(d/2)[h(\frac{2k}{d}) + 2h(1-\frac{2k}{d})] +(k-1)\log(\al)$ is negative
for $0 < \al < C_{k,d}$.

The proof is now complete since we have shown that if $2 \leq k \leq (1 - 1/\sqrt{2})d$ then
the derivative of (\ref{eqn:entterm}) is negative for $0 < \al < C_{k,d}$ .
\end{proof}

\begin{lem} \label{lem:boundent}
Suppose that $2 \leq k \leq (1-1/\sqrt{2})d$ and $0 \leq \al \leq 1$. Then the
entropy term (\ref{eqn:entterm}) is bounded from above by
$$\al ( k \log(d) + 1 ) +h(\al)(1-k) + (d/2)\al^3.$$
\end{lem}

\begin{proof}
We use the properties of $h(x)$ from (\ref{eqn:hproperty}). We have that
$h(\frac{2k}{d}\al) = \al h(\frac{2k}{d}) + \frac{2k}{d}h(\al)$,
$h(\al - \frac{2k}{d}\al) = \al h(1-\frac{2k}{d}) + (1-\frac{2k}{d})h(\al)$, and
$h(1-2\al + \frac{2k}{d}\al) \leq (2\al - \frac{2k}{d}\al) - \frac{1}{2}(2\al - \frac{2k}{d}\al)^2$.

Therefore,
\begin{align*}
&h(\al(2k/d)) + 2h(\al - \al(2k/d)) + h(1-2\al + \al(2k/d)) \; \leq \\
& \al \Big ( h(\frac{2k}{d}) + 2h(1-\frac{2k}{d}) + 2 - \frac{2k}{d} \Big ) + 2h(\al)(1 - \frac{k}{d}) - 2\al^2(1-\frac{k}{d})^2\,.
\end{align*}

Now, $H(\al, 1- \al) = h(\al) + h(1-\al) \geq h(\al) + \al -(1/2)\al^2 - (1/2)\al^3$ by property (3) of (\ref{eqn:hproperty}).
As a result (\ref{eqn:entterm}) is bounded from above by
\begin{equation} \label{eqn:entupperbound}
\al \Big[ \frac{d}{2}h(\frac{2k}{d}) + dh(1-\frac{2k}{d}) + 1-k\Big] -(k-1)h(\al) +
\al^2[\frac{d-1}{2}-d(1-\frac{k}{d})^2] + \frac{d-1}{2}\al^3.
\end{equation}

The term $\frac{d-1}{2}-d(1-\frac{k}{d})^2$ is increasing in $k$ and maximized when $k = (1-1/\sqrt{2})d$,
where it equals $-1/2$. Thus, $\al^2(\frac{d-1}{2}-d(1-\frac{k}{d})^2)$ is negative. The term
$ \frac{d}{2}h(\frac{2k}{d}) + dh(1-\frac{2k}{d}) + 1-k$ simplifies to $ k\log(d) -k\log(2k) + k + 1$, which
is at most $k\log(d) + 1$ because $k - k\log(2k) < 0$ if $k \geq 2$. Consequently, (\ref{eqn:entupperbound})
is bounded from above by $\al ( k \log(d) + 1 ) +h(\al)(1-k) + (d/2)\al^3$ as required.

\end{proof}

\paragraph{\textbf{Completion of the proof of Lemma} \ref{lem:ksparse}}

Recall that $N$ was defined to be the number of subsets $S \subset V(\G)$
of size at most $C_{k,d}\cdot n$ such that $S$ is not $k$-sparse. From (\ref{eqn:Nbound})
we have
$$\E{N} \leq dn^2 \max_{i \leq i \leq C_{k,d}n} \E{Z_{i,ki}}.$$

By Lemma \ref{lem:regexpectation}, $\E{Z_{i,ki}}$ is bounded from above by\\
$O(d\sqrt{n}) \times \exp{\left \{ n[(d/2)H(M(i/n,ki/nd)) - (d-1)H(i/n,1-(i/n))] \right \}}$.
Now,
\begin{align}
&\max_{1 \leq i \leq C_{k,d}n} (d/2)H(M(i/n,ki/nd)) - (d-1)H(i/n,1-(i/n)) \; \leq \nonumber \\
&\sup_{\frac{1}{n} \leq \al \leq C_{k,d}} (d/2)H(M(\al,(k/d)\al)) - (d-1)H(\al,1-\al), \label{eqn:entcontinuous}
\end{align}

where $\al$ is a continuous parameter. Lemma \ref{lem:decreasingent} shows that the supremum
of (\ref{eqn:entcontinuous}) is achieved at $\al = 1/n$ provided that $2 \leq k \leq (1-1/\sqrt{2})d$.
Lemma \ref{lem:boundent} implies that when $2 \leq k \leq (1-1/\sqrt{2})d$
the term in (\ref{eqn:entcontinuous}) is bounded from above at $\al = 1/n$ by
$\frac{1}{n}(k\log(d) + 1) + \frac{\log(n)}{n}(1-k) + \frac{d}{2n^3}$.
Therefore, we deduce that for $2 \leq k \leq (1-1/\sqrt{2})d$,
$$\E{N} \leq O(d^2n^{2.5}) \exp{\left \{n \big[ \frac{1}{n}(k \log(d) + 1) + \frac{\log(n)}{n}(1-k) + \frac{d}{2n^3} \big ] \right \}}.$$

If $n \geq \sqrt{d}$ then we see that $\E{N} \leq O(d^{k+2})n^{3.5-k}$. In particular,
if $k > 3.5$ then $\E{N} \to 0$ as $n \to \infty$. Hence, $\pr{N \geq 1} \leq \E{N} \to 0$
and this is precisely the statement of Lemma \ref{lem:ksparse}.

\subsection{Density of $k$-sparse graphs: proof of Lemma \ref{lem:densitybound}} \label{sec:regulardensitybound}

We begin with the following elementary lemma about the density of $k$-sparse sets.

\begin{lem} \label{lem:edgecountbound}
Let $S$ be a $k$-sparse set in a finite $d$-regular graph $G$. Then $|S|/|G| \leq \frac{d}{2d-2k}$.
\end{lem}

\begin{proof}
Set $|G| = n$, and so $|E(G)| = nd/2$. Consider the edge-profile $M(S)$ of $S$.
We have that $|S|/n = m(S,S) + m(S,S^{c})$.
Since $S$ is $k$-sparse, $m(S,S) \leq 2k|S|/(nd)$.
The number of edges from $S$ to $S^{c}$ is at most $d|S^{c}|$
because $\G$ is $d$-regular. Therefore, $m(S,S^{c}) \leq |S^{c}|/n$.
Consequently, $|S|/n \leq (\frac{2k}{d}-1)|S|/n + 1$, which implies
that $|S|/n \leq \frac{d}{2d-2k}$.
\end{proof}

Let $E$ denote the event that $\G$ contains an induced $k$-sparse subgraph of size $\al n$.
We bound the probability of $E$ by using the first moment method as well.
We will call a subset $S \subset V(\G)$ $k$-sparse if it induces a $k$-sparse subgraph.
By definition, any $k$-sparse set $S$ has the property that $e(S) \leq k |S|$.

Let $Z = Z(\al, \G)$ be the number of $k$-sparse sets in $\G$ of size $\al n$.
Recall the notation $Z_{i,j}$ from Section \ref{sec:ksparse}.
Let $Z_j = Z_{\al n,j}(\G)$ be the number of subsets $S \subset \G$ such that $|S| = \al n$
and the number of edges in $\G[S]$ is $j$. Then
\begin{equation} \label{eqn:Zbound}
\E{Z} = \sum_{j=0}^{k \al n} \E{Z_j}\,.
\end{equation}

From Lemma \ref{lem:zijexpectation} we see that $\E{Z_j}$ is of exponential order in $n$.
So the sum in (\ref{eqn:Zbound}) is dominated by the largest term.
From Lemma \ref{lem:zijexpectation} applied to $i = \al n$ and $j$ we conclude that
\begin{equation} \label{eqn:ziexpectation}
 \E{Z_j} = \dbinom{n}{\al n} \times
 \dfrac{(\al nd)! \; \big((1-\al)nd\big)! \; (nd/2)! \; 2^{\al nd -2j}}{(\al nd -2j)!\; j!\; \Big(\frac{(1-2\al)}{2}nd + j\Big)! \; (nd)!}.
\end{equation}

\begin{lem} \label{lem:maxexpectation}
If $\al > \frac{2k}{d}$ then the expectation of $Z_i$ is maximized at $i = k \al n$,
for all sufficiently large $n$. Note that $k \al n$ is the maximum number of edges
contained in a $k$-sparse set.
\end{lem}

\begin{proof}
We argue as in the proof of Lemma \ref{lem:Zijratio}.
From the equation for $\E{Z_j}$ in (\ref{eqn:ziexpectation}) we deduce that
\[\frac{\E{Z_{j+1}}}{\E{Z_j}} = \dfrac{(\al nd -2j-1)(\al nd -2j)}{4 (j+1)\Big(\frac{(1-2\al)}{2}nd + j + 1\Big)}\,.\]

This ratio is at least 1 for all $0 \leq j \leq k \al n$ if $n$ is sufficiently large and $\al > 2k/d$.
Indeed, subtracting the denominator from the numerator gives
$\al nd(\al nd-1) -2(1-2\al)nd - 4 - 2j(nd + 3)$. This is non-negative for all $0 \leq j \leq k \al n$ if and only if
\begin{equation} \label{eqn:ratiobound2}
k \al n \leq \dfrac{\frac{1}{2}(\al nd)(\al nd -1) - (1-2\al)nd - 2}{nd + 3}\,.
\end{equation}

If the inequality in (\ref{eqn:ratiobound2}) fails to hold for all sufficiently large $n$ then after dividing through by $n$ and 
letting $n \to \infty$ we conclude that $k \al \geq (1/2)\al^2 d$. This implies that $\al \leq 2k/d$, which contradicts our assumption.

\end{proof}

From Lemma \ref{lem:regexpectation} applied to $\E{Z_{i,j}}$ for $i = \al n$ and $j = k \al n$ we conclude that
\begin{equation} \label{eqn:regexpectation}
\E{Z_j} \leq O(\sqrt{n}) \, \exp{\left \{ n \left [ \frac{d}{2}H\Big(M(\al,j/nd)\Big) - (d-1)H\big(\al,1-\al\big) \right ]\right \}}.
\end{equation}

For the rest of this section we assume that $\al \geq (\log d)/d$ and $d$ is large enough
such that $(\log d)/d > 2k/d$. This will hold since $k = o(\log d)$. If $\al < (\log d)/d$ then there is nothing to prove.
We conclude from Lemma \ref{lem:maxexpectation}, (\ref{eqn:regexpectation}) and (\ref{eqn:Zbound}) that
\begin{align}
\E{Z} & \leq (kn) \, \E{Z_{k \al n}} \nonumber \\
 & \leq O(kn^{3/2}) \, \exp{\left \{ n \left [ \frac{d}{2}H\Big(M(\al,\frac{k}{d} \al)\Big) - (d-1)H\big(\al,1-\al\big) \right ]\right \}}. \label{eqn:Zmax}
\end{align}

Note that $M(\al, \frac{k}{d} \al)$ equals
$$ M(\al, \frac{k}{d} \al) =  \left [ \begin{array}{cc}
 \frac{2k \al}{d} & \al - \frac{2k \al}{d}\\
\al - \frac{2k \al}{d} & 1-2\al + \frac{2k \al}{d}
\end{array} \right ].
$$
This matrix may depend on $n$ through $\alpha$. If it does then we replace $\al$ by its limit supremum as $n \to \infty$.
By an abuse of notation we denote the limit supremum by $\al$ as well.

For $d \geq 3$ define $\al_d = \al_{d,k}$ by
$$
\al_d = \sup \big \{\al : 0 \leq \al \leq 1 \;\text{and}\;\; \frac{d}{2} H\big(\M(\al,\frac{k}{d}\al)\big) - (d-1)H(\al,1-\al) \geq 0 \big \}.
$$

Thus, if $\al > \al_d$ then from the continuity of the entropy function $H$ we conclude that for all
sufficiently large $n$ the function $\frac{d}{2}H\Big(M(\al,\frac{k}{d} \al)\Big) - (d-1)H\big(\al,1-\al\big) < 0$.
Consequently, from (\ref{eqn:Zmax}) we conclude that $\limsup_{n \to \infty} \pr{E} \leq \limsup_{n \to \infty} \E{Z} = 0$.
We devote the rest of this section to bounding the entropy functional in order to show that
$\al_d \leq (2+\eps) \frlog$ for all large $d$.

First, we show that $\al_d \to 0$ as $d \to \infty$. Suppose otherwise, that $\limsup_{d \to \infty} \al_d = \al_{\infty} > 0$.
Lemma \ref{lem:edgecountbound} implies that $\al_{\infty} \leq 1/2$ because $\al_d \leq d/(2d-2k)$ and $k = o(\log d)$.
After passing to an appropriate subsequence in $d$, noting that $2k/d \to 0$ as $d \to \infty$ due to $k = o(\log d)$,
and using the continuity of $H$ we see that
\[ \lim_{d \to \infty} \, \frac{1}{2}H\big(\M(\al_d, \frac{k}{d} \al_d)\big) - H(\al_d, 1 - \al_d) =
\frac{1}{2}H\big(M(\al_{\infty},0)\big) - H(\al_{\infty}, 1 - \al_{\infty})\,.\]

However, $(1/2)H(M(x,0)) - H(x,1-x) = (1-x)\log(1-x) -(1/2)(1-2x)\log(1-2x)$, and this is negative for $0 < x \leq 1/2$.
This can be seen by noting that the derivative of the expression is negative for $x > 0$ and the expression vanishes
at $x=0$. Therefore, for all large $d$ along the chosen subsequence we have
$\frac{d}{2}H\big(\M(\al_d,\frac{k}{d} \al_d)\big) - (d-1)H(\al_d, 1 - \al_d)  < 0$; a contradiction.

We now analyze the supremum of the entropy functional for large $d$ in order to bound $\al_d$.
From the properties of $h(x)$ in (\ref{eqn:hproperty}) we deduce that
\begin{align}
H\big(\M(\al)\big) &= h\left (\frac{2k \al}{d} \right ) + 2h\left (\al - \frac{2k \al}{d} \right ) +
h\left (1-2\al + \frac{2k \al}{d} \right ) \nonumber \\
&\leq 2[h(\al) + \al -\al^2] + \frac{2k}{d}[\al -h(\al) + \al \log (\frac{d}{2k}) + 2\al^2] \,, \label{eqn:matentropy} \\
H(\al, 1-\al) &= h(\al) + \al - \frac{1}{2}\al^2 + O(\al^3). \label{eqn:vecentropy}
\end{align}

From (\ref{eqn:matentropy}) and (\ref{eqn:vecentropy}) we see that $\frac{d}{2}H(\M(\al)) -(d-1)H(\al,1-\al)$ is at most
\begin{equation} \label{eqn:entropyasymp}
 - \frac{d}{2}\al^2 + k[\al -h(\al) + \al \log \big (\frac{d}{2k}\big) + 2\al^2 ] + \al + h(\al) + O(d\al^3).
\end{equation}

Now, $k(\al + 2\al^2) + \al \leq 4k \al$ and $\log(d/2k) \leq \log (d/k)$. Hence, (\ref{eqn:entropyasymp}) is at most
\begin{equation} \label{eqn:entropyasymp2}
 - \frac{d}{2}\al^2 + k[\al \log (d/k) - h(\al)] + h(\al) + 4k\al + O(d\al^3).
 \end{equation}

Let us write $\al = \beta \frlog$ where $\beta \geq 1$.
In terms of $\beta$, $h(\al) = \beta \frac{\log^2 d - \log d \log\log d}{d} + h(\beta)\frlog$.
Since $\beta \geq 1$, $h(\beta) \leq 0$, and we get that
$- \frac{d}{2}\al^2 + h(\al) \leq (-\frac{\beta^2}{2} + \beta) \frac{\log^2 d}{d}$.
The term $\al \log (d/k) - h(\al)$ equals
$ \beta \, \frac{\log d( \log\log d - \log k) }{d} + \beta \log \beta \frlog$.
Substituting $k = \eps_d \log d$ and combining these inequalities together
we see that (\ref{eqn:entropyasymp2}) is bounded from above by
\begin{equation} \label{eqn:entropyasymp3}
[1 - (1/2)\beta +\eps_d \log \beta -\eps_d\log(\eps_d)+ 4\eps_d]\,\beta\frac{\log^2d}{d} + O(\frac{\beta^3 \log^3 d}{d^2}).
\end{equation}

As $\beta = o_d(d/(\log d))$ the term $\beta^3 (\log^3 d)/d^2$ is of order
$o_d(\beta^2 (\log^2 d)/d)$ as $d \to \infty$. Therefore, (\ref{eqn:entropyasymp3}) is of the form
$1 - (1/2 - o_d(1))\beta + \eps_d \log \beta -\eps_d\log(\eps_d)+ 4\eps_d$ for large $d$.
Elementary calculus shows that in order for $1 - [(1/2) - \delta]\beta + \delta \log \beta - \delta \log(\delta) + C \delta$
to be non-negative $\beta$ must satisfy $\beta \leq 2 - 2\delta \log(\delta) + 4C\delta$, provided that $0 \leq \delta \leq 1$.

We conclude that there is a function $\delta(d) = \delta(\eps_d)$ such that $\delta(d) \to 0$ as $d \to \infty$
and (\ref{eqn:entropyasymp3}) is negative unless $\beta \leq 2 + \delta(d)$. As a result, we have
$\al_d \leq (2 + \delta(d)) \frlog$ and the latter is bounded by $(2+\eps) \frlog$ for all large $d$.
This completes the proof of Lemma \ref{lem:densitybound}.

\section{Percolation on Erd\H{o}s-R\'{e}nyi graphs} \label{sec:ER}

\begin{lem} \label{lem:cycle}
The expected number of cycles of length no more than $\tau$ in $\ER$ is at most $d^\tau \log \tau$.
\end{lem}

\begin{proof}
Let $C_{\ell}$ denote the number of cycles of length $\ell \geq 3$ in $\ER$.
Note that $\E{C_{\ell}} = \binom{n}{\ell} \frac{\ell!}{2\ell} (d/n)^{\ell}$, and
$\binom{n}{\ell} \frac{\ell!}{2\ell} (d/n)^{\ell} \leq \frac{d^{\ell}}{2\ell}$. The number of
cycles of length at most $\tau$ is $C_{\leq \tau} = C_3 + \cdots + C_{\tau}$.
Note that $\sum_{\ell = 3}^{\tau} 1/(2\ell) \leq \int_2^{\tau} \frac{1}{t} \,dt = \log(\tau/2) \leq \log \tau$.
Thus,
\begin{equation*}
\E{C_{\leq \tau}} = \sum_{\ell =3}^{\tau} \E{C_{\ell}}
 \leq \sum_{\ell =3}^{\tau} \frac{d^{\ell}}{2\ell}
 \leq d^{\tau} \log \tau.
\end{equation*}

\end{proof}

Let $X_{n,\tau}$ be the number of cycles of length at most $\tau$ in $\ER$. It follows from
Lemma \ref{lem:cycle} that if $\tau = \log_d(n) - \log \log \log (n) - \log (\omega_n)$ then
$\E{X_{n,\tau}} = O(n/\omega_n)$.

\subsection{Proof of Theorem \ref{thm:ERpercolation}} \label{sec2.1}
Let $E$ denote the event that $\ER$ contains a percolation
set of size $\al n$ with clusters of size at most $\tu$. We can assume that $\al > (2e)/d$, for otherwise, there
is nothing to prove due to $d \geq 5$. We bound the probability of $E$ by using the
first moment method. From this we will show that if $\al n$ is bigger than the bound in the statement
of Theorem \ref{thm:ERpercolation} then $\pr{E} \to 0$ as $n \to \infty$. 

Set $\mu_n = \E{X_{n,\tau}} = O(n/\omega_n)$ for $\tau$ in the statement of Theorem \ref{thm:ERpercolation}.
Fix $\delta > 0$ and note that $\pr{X_{n,\tau} \geq \mu_n/\delta} \leq \delta$ from Markov's inequality.

Let $Z = Z(\al, \ER)$ be the number of percolation sets in $\ER$ of size $\al n$ with clusters of
size at most $\tu$. From the observation above we have that 
\begin{equation} \label{eqn:Ebound}
\pr{E} \leq \pr{E \cap \{X_{n,\tau} \leq \mu_n/\delta\}} + \delta \leq \E{Z; X_{n,\tau} \leq \mu_n/\delta} + \delta
\end{equation}

where $\E{Z; X_{n,\tau} \leq \mu_n/\delta}$ denotes the expectation of $Z$ on the event
$\{X_{n,\tau} \leq \mu_n/\delta\}$. To prove the theorem it suffices to show that for any $\delta > 0$
the expectation $\E{Z; X_{n,\tau} \leq \mu_n/\delta}$ vanishes to zero as $n \to \infty$ provided that
$\al n$ is bigger than the bound stated in the statement of Theorem \ref{thm:ERpercolation}.
For then we have that $\limsup_{n \to \infty} \pr{E} \leq \delta$ for any $\delta > 0$, and thus, $\pr{E} \to 0$.

We now make a crucial observation about percolation sets with small clusters.
Let $S$ be a percolation set with clusters of size at most $\tu$. If we remove an edge from
every cycle of the induced graph $\mathrm{ER}[S]$ of length at most $\tu$ then the
components of $\mathrm{ER}[S]$ become trees. In that case the number of remaining edges in
$\mathrm{ER}[S]$ is at most $|S|$. Therefore, the number of edges in $\mathrm{ER}[S]$ is at most
$|S| + (\mu_n/\delta)$. This bound is useful as it shows that the subgraph included by percolation
sets with small clusters is much more sparse relative to the original graph.

Let $M = M(\al,\tau, \delta, \ER)$ be the number of subsets $S \subset \ER$ such that
$|S| = \al n$ and the number of edges in $\mathrm{ER}[S]$ is at most $|S| + (\mu_n/\delta)$.
Notice that the number of edges in $\mathrm{ER}[S]$ is distributed as the binomial random
variable $\mathrm{Bin}(\binom{|S|}{2},d/n)$. The observation above implies that
\begin{align}
\E{Z; X_{n,\tau} \leq \mu_n/\delta} &\leq \E{M} \label{eqn:Zbound2} \\
&= \binom{n}{\al n} \, \pr{\mathrm{Bin} \left (\binom{\al n}{2},d/n \right ) \leq \al n + (\mu_n/\delta)}. \nonumber
\end{align}

\begin{lem} \label{lem:binomialbound}
Let $\mathrm{Bin}(m,p)$ denote a binomial random variable with parameters $m \geq 1$ and $0 \leq p \leq 1$.
If $ 0 < p \leq 1/2$ and $0 < \mu \leq 1$ then the following bound holds.
\[ \pr{\mathrm{Bin}(m,p) \leq \mu mp} = O(\sqrt{mp}) \times \exp{\left \{ -m \, \big [ \mu p \log \mu + (1-\mu)p - \mu p^2 \big ] \right \}}.\]
\end{lem}

\begin{proof}
The quantities $\pr{\mathrm{Bin}(m,p) = k} = \binom{m}{k} p^k (1-p)^{m-k}$ are non-decreasing
in $k$ if $k \leq mp$. Therefore,
$$ \pr{\mathrm{Bin}(m,p) \leq \mu mp} =
\pr{\mathrm{Bin}(m,p) \leq \lfloor \mu mp \rfloor} \leq \mu mp \, \pr{\mathrm{Bin}(m,p) = \lfloor \mu mp \rfloor}.$$

We can estimate $\pr{\mathrm{Bin}(m,p) = \lfloor \mu mp \rfloor}$ by $\binom{m}{\mu mp} p^{\mu mp}(1-p)^{m-\mu mp}$
with a multiplicative error term of constant order. Stirling's approximation implies
$\binom{m}{\mu mp}$ is bounded from above by $O((m\mu p (1-\mu p))^{-1/2}) \,e^{mH(\mu p)}$.
Therefore, after some algebraic simplifications and using $1-\mu p \geq 1/2$ we deduce that
\begin{equation} \label{eqn:binomprobbound}
\pr{\mathrm{Bin}(m,p) \leq \mu mp} \leq O((mp)^{-1/2}) \, e^{m[-\mu p\log \mu -(1-\mu p)(\log(1-\mu p) -\log(1-p))]}\,.
\end{equation}

We now provide an upper bound to the exponent on the r.h.s.~of (\ref{eqn:binomprobbound}).
Note that $x \leq -\log(1-x) \leq x + x^2$ for $0 \leq x \leq 1/2$. As $\mu p \leq p < 1/2$, it follows
from these two inequalities that $\log(1-\mu p) -\log(1-p) \geq p -\mu p - \mu^2 p^2$. Hence,
$$-\mu p\log \mu -(1-\mu p)(\log(1-\mu p) -\log(1-p)) \leq -\mu p \log \mu -(1-\mu)p + \mu p^2\,.$$
The conclusion of the lemma follows upon substituting the bound above into the exponent
on the r.h.s.~of (\ref{eqn:binomprobbound}).

\end{proof}

We now use Lemma \ref{lem:binomialbound} to provide an upper bound to
$\pr{\mathrm{Bin}(\binom{\al n}{2},d/n) \leq \al n + (\mu_n/\delta)}$. We require
that $n \geq 2d$ and have that $\mu = [\al n + (\mu_n/\delta)]\big /[\binom{\al n}{2}(d/n)]$.

Recall that $\al > 2e/d$. With this assumption and for $n \geq 2d$ it is easy to show that
$\mu \leq 2/(d\al) + O_d(1/\omega_n)$. For all large $n$ we thus have $\mu \leq e^{-1}$.
From Lemma \ref{lem:binomialbound} we deduce:
\begin{equation} \label{eqn:ERedgebound}
\pr{\mathrm{Bin} \left (\binom{\al n}{2},d/n \right ) \leq \al n + (\mu_n/\delta)} \leq
O(\sqrt{nd}) \, e^{-\binom{\al n}{2}\frac{d}{n} \,[\mu \log(\mu) + 1 - \mu - \frac{d}{n}]}.
\end{equation}

We now simplify the exponent in (\ref{eqn:ERedgebound}). The function $x \to x\log x$ is decreasing
for $0 \leq x \leq e^{-1}$. Hence, as $\mu \leq 2/(d \al) + O_d(1/\omega_n) \leq e^{-1}$, we have
$\mu \log(\mu) \geq (\frac{2}{d \al} + O_d(1/\omega_n)) \log \big (\frac{2}{d \al} + O_d(1/\omega_n) \big)$.
From this lower bound on $\mu \log (\mu)$ it follows easily that
$$\mu \log(\mu) + 1 - \mu \geq \frac{2}{d \al} \log \big (\frac{2}{d \al}\big) + 1 - \frac{2}{d \al} - O_d(1/\omega_n).$$
Also, $\binom{\al n}{2}(d/n) \geq \frac{\al^2 d}{2}n - O_d(1)$. Combining these estimates we gather that
the exponent in (\ref{eqn:ERedgebound}) is bounded from above by
\begin{equation} \label{eqn:ERedgeexponent}
-n \left( \al \log \big ( \frac{2}{\al d} \big) + \frac{\al^2d}{2} - \al \right) + O_d \big ( \max \{n/\omega_n, 1\} \big).
\end{equation}

Now we can provide an upper bound to $\E{M}$ from (\ref{eqn:Zbound2}).
Stirling's approximation implies $\binom{n}{\al n} \leq 2 e^{nH(\al, 1-\al)}$.
Combining this with the bound on the binomial probability that is on the r.h.s.~of (\ref{eqn:Zbound2}),
derived from the inequalities in (\ref{eqn:ERedgebound}) and (\ref{eqn:ERedgeexponent}), we have
$$\E{M} \leq O_d(\sqrt{n}) \, e^{n \left [H(\al,1-\al) + \al \log ( \frac{2}{\al d} ) + \frac{\al^2d}{2} - \al \right ]
+ O_d(\max \{n/\omega_n, 1\})}.$$

Now, $H(\al,1-\al) + \al \log ( \frac{2}{\al d} ) + \frac{\al^2d}{2} - \al = h(1-\al) + \al(1 + \log(d/2)) - (d/2)\al^2$.
From (3) of (\ref{eqn:hproperty}) we have $h(1-\al) \leq \al$. Consequently,
$h(1-\al) + \al(1 + \log(d/2)) - (d/2)\al^2 \leq \al (2 + \log(d/2) - (d/2)\al)$.
This implies that
$$ \E{M} \leq O_d(\sqrt{n}) \, e^{\al n \big [2 + \log(d/2) - (d/2)\al \big ] + O_d(\max \{n/\omega_n, 1\})}.$$

From (\ref{eqn:Ebound}) and (\ref{eqn:Zbound2}) we have
$\pr{E} \leq \E{Z; X_{n,\tau} \leq \mu_n/\delta} + \delta \leq \E{M} + \delta$, and thus,
\begin{equation*}
\pr{E} \leq O_d(\sqrt{n}) \, e^{\al n \big [2 + \log(d/2) - (d/2)\al \big ] + O_d(\max \{n/\omega_n, 1\})} + \delta.
\end{equation*}

If $2 + \log(d/2) - (d/2)\al < 0$ then $\limsup_{n \to \infty} \pr{E} \leq \delta$ for all $\delta > 0$.
This implies $\pr{E} \to 0$ as $n \to \infty$, and thus, with high probability $\ER$ does not contain
induced subgraphs of size larger than $\al n$ such that their components have size at most
$\tau = \log_d(n) - \log\log\log(n) - \log(\omega_n)$. The condition $2 + \log(d/2) - (d/2)\al < 0$ is equivalent
to $\al > \frac{2}{d}(\log d + 2 - \log 2)$, which is precisely the bound in the statement of Theorem \ref{thm:ERpercolation}.

\section*{Acknowledgements}
The author thanks B\'{a}lint Vir\'{a}g for suggesting the problem.


\begin{thebibliography}{10}

\bibitem{BKZZ} J.~Barbier, F.~Krzakala, L.~Zdeborova, and P.~Zhang,
The hard-core model on random graphs revisited,
\emph{J. Phys.: Conf. Ser.} \textbf{473} 012021 (2013), \arxiv{1306.4121}.

\bibitem{BWZ} S.~Bau, N.C.~Wormald, and S.~Zhou,
Decycling number of random regular graphs,
\emph{Random Structures \& Algorithms} \textbf{21} (2002), pp.~397--413.

\bibitem{BGT} M.~Bayati, D.~Gamarnik, and P.~Tetali,
Combinatorial approach to the interpolation method and scaling limits in sparse random graphs,
\emph{Annals of Probability} \textbf{41} (2013), pp.~4080--4115, \arxiv{0912.2444}.

\bibitem{Bol} B.~Bollob\'{a}s,
The independence ratio of regular graphs,
\emph{Proc. Amer. Math. Soc.} \textbf{83} no. 2 (1981), pp.~433--436.

\bibitem{Bol2} B.~Bollob\'{a}s,
A probabilistic proof of an asymptotic formula for the number of labelled regular graphs,
\emph{European Journal of Combinatorics} \textbf{1} (1980), pp.~311--316.

\bibitem{Bolbook} B.~Bollob\'{a}s,
Random graphs, 2nd ed.,
Cambridge University Press, 2001.

\bibitem{BC} E.~A.~Bender and E.~R.~Canfield,
The asymptotic number of labelled graphs with given degree sequences,
\emph{Journal of Combinatorial Theory Series A} \textbf{24} (1978), pp.~296--307.

\bibitem{CGHV} E.~Cs\'{o}ka, B.~Gerencs\'{e}r, V.~Harangi, and B.~Vir\'{a}g,
Invariant Gaussian processes and independent sets on regular graphs of large girth,
to appear in \emph{Random Structures \& Algorithms} (2015), \arxiv{1305.3977}.

\bibitem{EF} K.~Edwards and G.~Farr,
Fragmentability of graphs,
\emph{Journal of Combinatorial Theory Series B} \textbf{82} (2001), pp.~30--37.

\bibitem{Fri} A.M.~Frieze,
On the independence number of random graphs,
\emph{Discrete Mathematics} \textbf{81} (1990), pp.~171--175.

\bibitem{FL} A.M.~Frieze and T.~{\L}uczak,
On the independence and chromatic numbers of random regular graphs,
\emph{Journal of Combinatorial Theory Series B} \textbf{54} (1992), pp.~123--132.

\bibitem{HW08} C.~Hoppen and N.~Wormald,
Induced forests in regular graphs with large girth,
\emph{Combinatorics, Probability and Computing} \textbf{17}(3) (2008), pp.~389--410.

\bibitem{HW} C.~Hoppen and N.~Worlmald,
Local algorithms, regular graphs of large girth, and random regular graphs,
\emph{preprint} (2013), \arxiv{1308.0266}.

\bibitem{McKay} B.D.~McKay,
Independent sets in regular graphs of high girth,
\emph{Ars Combinatorica} \textbf{23}A (1987), pp.~179--185.

\end{thebibliography}
\end{document}